\documentclass[12pt]{amsart}
\usepackage{amsmath, amscd, amssymb, amsthm,amsfonts}
\usepackage[nohug]{diagrams}
\usepackage[T1]{fontenc}
\makeatletter

\usepackage{dbnsymb}
\usepackage{mthsymb7}
\usepackage{stmaryrd}

\ifx\pdfoutput\undefined
\usepackage{graphicx}
\else
\usepackage[pdftex]{graphicx}
\usepackage{epstopdf}
\fi

\newcommand{\co}{\colon\thinspace}

\DeclareMathOperator{\tr}{tr}

\DeclareMathOperator{\Cob}{Cob}
\DeclareMathOperator{\PCob}{Pre-Cob}

\DeclareMathOperator{\Endo}{End}
\DeclareMathOperator{\Star}{St}

\DeclareMathOperator{\Kom}{Kom}
\DeclareMathOperator{\Mat}{Mat}

\DeclareMathOperator{\Morph}{Hom}
\DeclareMathOperator{\Cone}{Cone}

\DeclareMathOperator{\Univ}{U}
\DeclareMathOperator{\TL}{TL}
\DeclareMathOperator{\BMW}{BMW}
\DeclareMathOperator{\K}{K}

\DeclareMathOperator{\SU}{SU}
\DeclareMathOperator{\SO}{SO}

\DeclareMathOperator{\Homology}{H}

\newcommand{\inp}[1]{\ensuremath{\langle #1 \rangle}}

\newcommand{\normaltext}[1]{\textnormal{#1}}
\newcommand{\quantsu}[0]{\Univ_q\mathfrak{su}(2)}

\newcommand{\cTrace}{\CPic{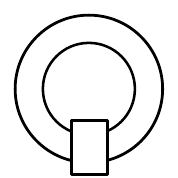}}

\newcommand{\inv}{^{-1}}

\newcommand{\Z}{\mathbb{Z}}

\newcommand{\hsaddle}{\myhsaddle}
\newcommand{\isaddle}{\myisaddle}
\newcommand{\projector}{\twoprojector}

\newcommand{\CPic}[1]{
\begin{minipage}{.45in}
\includegraphics[scale=.75]{#1}
\end{minipage}
}

\newcommand{\CPPic}[1]{
\begin{minipage}{.8in}
\includegraphics[scale=1.1]{#1}
\end{minipage}
}

\newcommand{\MPic}[1]{
\begin{minipage}{.35in}
\includegraphics[scale=.45]{#1}
\end{minipage}
}

\newcommand{\BPic}[1]{
\begin{minipage}{1in}
\includegraphics[scale=1.5]{#1}
\end{minipage}
}

\newcommand{\Cpic}{\CPic}

\newcommand{\Mpic}{\MPic}

\theoremstyle{plain}
\newtheorem{theorem}[subsection]{Theorem}

\newtheorem*{conjecture}{Conjecture}
\newtheorem{proposition}[subsection]{Proposition}
\newtheorem{corollary}[subsection]{Corollary}
\newtheorem{lemma}[subsection]{Lemma}
\theoremstyle{remark}

\theoremstyle{definition}

\newtheorem{definition}[subsection]{Definition}

\numberwithin{equation}{section}

\begin{document}
\title[$\SO(3)$ Homology of Graphs and Links]{${\mathbf{SO(3)}}$ Homology of Graphs and Links}

\author[Benjamin Cooper, Matt Hogancamp and Vyacheslav Krushkal]{Benjamin Cooper, Matt Hogancamp and Vyacheslav Krushkal}

\address{Department of Mathematics, University of Virginia, Charlottesville, VA 22904}
\email{bjc4n\char 64 virginia.edu, mhoganca\char 64 gmail.com, krushkal\char 64 virginia.edu}

\begin{abstract}
The $\SO(3)$ Kauffman polynomial and the chromatic polynomial of planar
graphs are categorified by a unique extension of the Khovanov homology
framework. Many structural observations and computations of homologies of
knots and spin networks are included.
\end{abstract}

\maketitle

\section{Introduction}

In \cite{KH} Mikhail Khovanov introduced a categorification of the
Temperley-Lieb algebra. Recently, two of the authors \cite{CK} have shown
that there are chain complexes within this construction that become the
Jones-Wenzl projectors in the image of the Grothendieck group $\K_0$. These
chain complexes are unique up to homotopy and idempotent with respect to the
tensor product: $C\otimes C\simeq C$.  It is now well-known \cite{FK} that
the chromatic algebra and the $\SO(3)$ Birman-Murakami-Wenzl algebra can be
constructed using the second Jones-Wenzl projector. In this paper we use the
formulation of Bar-Natan \cite{DBN} to extend the original categorification
of the Temperley-Lieb algebra to categorifications of the $\SO(3)$ BMW
algebra and the chromatic algebra. Previous work on the categorification of
the chromatic polynomial \cite{GR,Stosic} has been focused on constructions
which are in many respects independent of structural choices such as the
Frobenius algebra. In this paper we obtain an essentially unique
categorification of the chromatic polynomial of planar graphs.

We begin by interpreting the second Jones-Wenzl projector in the
Temperley-Lieb algebra as an algebra of $q$-power series with
$\Z$-coefficients,
$$p_2\;=\;\smoothing\, -\, \frac{1}{q+q^{-1}}\;\; \hsmoothing\;=\; \smoothing\, +\, \sum_{i=1}^{\infty}(-1)^i q^{2i-1}\; \;\hsmoothing$$
This power series is replaced by a chain complex in the categorification
which is then shown to satisfy uniqueness and idempotence
properties up to homotopy. While the categorification of the Jones-Wenzl projectors $p_n$
for all $n$ is presented in \cite{CK}, in this paper we give a self-contained account for the
second projector.  Using this chain complex the 2-categorical
``canopolis'' structure of the Khovanov categorification then extends from a
categorification of the Temperley-Lieb planar algebra to a categorification
of the $\SO(3)$ BMW algebra and chromatic algebra. It is checked that the
local relations in these algebras are satisfied up to homotopy by our
construction.

We conclude with a number of calculations of homologies of links and spin networks and
some preliminary observations about the structure of the space of morphisms.
Two explicit calculations are included in order to demonstrate the ease with
which our model lends itself to calculation. We include the chromatic
homology for tree and generalized theta graphs. The homology of the sheet algebra is
computed and we conjecture that all graph homology is structured in a
specific way. Due to the universal nature of the construction in \cite{CK}
the authors believe that these calculations will agree with those made using
other frameworks for the categorification of representation theory.

\section{Diagrammatic Algebras} \label{TL section}

This section summarizes the relevant background on definitions of the
Temperley-Lieb algebra, the chromatic algebra and the $\SO(3)$ BMW
algebra, and on the relations between them.

\subsection{Temperley-Lieb Algebra} \label{TL subsection}  The Temperley-Lieb algebra $\TL_n$ is the
$\mathbb{Z}[q,q^{-1}]$-algebra determined by subjecting the generators $1$,
$e_1$, $e_2,\ldots, e_{n-1}$ to the relations:

\begin{enumerate}
\item $e_i \cdot e_j = e_j\cdot e_i$ if $|i-j| \geq 2$.
\item $e_i\cdot e_{i\pm 1}\cdot e_i = e_i$
\item $e_i^2 = -[2] e_i$
\end{enumerate}

where the quantum integer $[2] = q + q^{-1}$.

Each generator $e_i$ can be pictured as a diagram consisting of $n$ chords
between two collections of $n$ points on two horizontal lines in the plane.
All strands are vertical except for two, connecting the $i$th and the
$(i+1)$-st points in each collection. For instance, when $n = 3$ we have the
following diagrams:
$$1 = \CPic{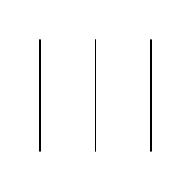}\hspace{.05in},\quad e_1 = \CPic{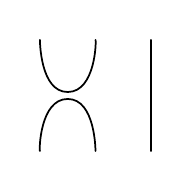} \textnormal{\quad and\quad } e_2 = \CPic{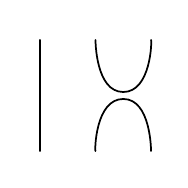} $$

The multiplication is given by vertical composition of diagrams. Planar
isotopy induces relations 1 and 2 between the generators above while the
third relation states that a disjoint circle evaluates to $-q-q^{-1}$.

This algebra is well-known in low-dimensional topology due to the extension
from planar diagrams to tangles given by the Kauffman bracket relations:
$$\CPic{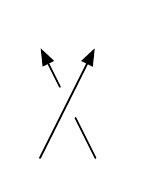} = q \CPic{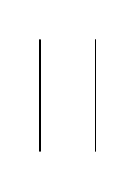} -\hspace{.2in} q^{2} \CPic{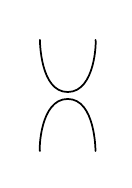} \normaltext{\quad and \quad} \CPic{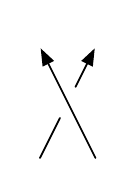} = q^{-2} \CPic{n2-s} -\hspace{.2in} q^{-1}\CPic{n2-1} .$$

$\TL_n$ is included into $\TL_{n+1}$ by adding a vertical strand on the
right, and $\TL$ is defined to be $\cup_n \TL_n$. The trace $\tr_{\TL}\co
\TL_n\longrightarrow {\mathbb{Z}[q,q^{-1}]}$ is defined on the additive generators
(rectangular pictures) by connecting the top and bottom endpoints by
disjoint arcs in the complement of the rectangle in the plane. The result is
a disjoint collection of circles in the plane, which are then evaluated by
taking $(q+q^{-1})^{\# circles}$.

\begin{definition}{(Jones-Wenzl projector)}\label{jwproj tl}
There is a special element $p_2 \in \TL_{2}$ (where the coefficients are taken to be rational functions of the variable $q$),
$$p_2 = 1 - \frac{1}{q+q^{-1}} e_1,$$

called the second \emph{Jones-Wenzl projector}. Graphically,
$$\CPic{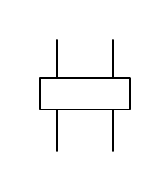} = \CPic{n2-1} - \frac{1}{q+q^{-1}} \CPic{n2-s}$$
\end{definition}

The second Jones-Wenzl projector $p_2$ satisfies the properties

\begin{enumerate}
\item $p_2 \cdot e_1 = 0 = e_1\cdot p_2$
\item $p_2 \cdot p_2 = p_2$
\end{enumerate}

In representation theory, the Temperley-Lieb algebra is the algebra of
$\quantsu$-equivariant maps between $n$-fold tensor powers of the
fundamental representation $V$:
$$ \TL_n = \Morph_{\quantsu}(V^{\otimes n},V^{\otimes n}). $$

The subalgebra determined by the projector $p_2$ corresponds to the second
irreducible representation of $\quantsu$. The second irreducible
representation of $\SU(2)$ is the fundamental representation of $\SO(3)$.

\subsection{The ${\mathbf{SO(3)}}$ BMW algebra}\label{so bmw sec}
We review some background material on the $\SO(N)$ Birman-Murakami-Wenzl
algebra; see \cite{BW,M} for more details.  $\BMW(N)_n$ is the algebra of
framed tangles on $n$ strands in $D^2\times [0,1]$ modulo regular isotopy and the
$\SO(N)$ Kauffman skein relations:

\begin{align*}
\CPic{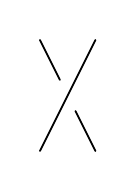} - \CPic{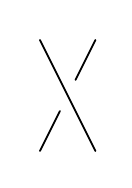} = (q^2 - q^{-2})\Bigg(\CPic{n2-1} - \CPic{n2-s}\Bigg), \\
\CPic{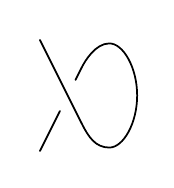} = q^{2(N-1)}\CPic{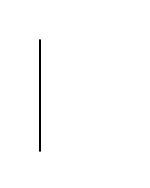} \normaltext{\quad and\quad} \CPic{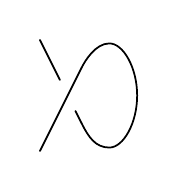} = q^{2(1-N)}\CPic{line}.
\end{align*}

By a tangle we mean a collection of curves (some of them perhaps closed)
embedded in $D^2\times [0,1]$, with precisely $2n$ endpoints, $n$ in
$D^2\times\{0\}$ and $D^2\times\{1\}$ each, at the prescribed marked points
in the disk. The tangles are framed, i.e. they are given with a
trivialization of their normal bundle. This is necessary since the $q^{\pm
  2(1-N)}$-skewed versions of the first Reidemeister move in the Kauffman
relations above are inconsistent with invariance under the first
Reidemeister move. As with $\TL_n$, the multiplication is given by vertical
stacking. Like above, $\BMW(N)=\cup_n \BMW(N)_n$.

The Markov trace $\hbox{tr}_K\co \BMW(N)_n\longrightarrow {\mathbb
  Z}[q,q^{-1}]$ is defined on the generators by connecting the top and
bottom endpoints by standard parallel arcs in the complement of $D^2\times[0,1]$ in
$3$-space, sweeping from top to bottom, and computing the $\SO(N)$ Kauffman
polynomial (using the above skein relations) of the resulting link. Below we
will discuss this trace in detail.

Since the object of main interest in this paper is the $\SO(3)$ algebra, we
will omit $N=3$ from the notation, and set $\BMW_n = \BMW(3)_n$.

\subsection{The chromatic polynomial and the chromatic algebra} \label{chromaticpolydef}
The {\em chromatic polynomial} ${\chi}_{\Gamma}(Q)$ of a graph $\Gamma$, for
$Q\in {\mathbb Z}_+$, is the number of colorings of the vertices of $\Gamma$
with the colors $1,\ldots, Q$ where no two adjacent vertices have the same
color. To study ${\chi}_{\Gamma}(Q)$ for non-integer values of $Q$, it is
convenient to use the contraction-deletion relation. Given any edge $e$ of
$\Gamma$ which is not a loop,
\begin{equation}
\label{chromatic poly1}
\chi_{\Gamma}^{}(Q)={\chi}_{{\Gamma}\backslash e}(Q)-{\chi}_{{\Gamma}/e}(Q)
\end{equation}

where ${\Gamma}\backslash e$ is the graph obtained from $\Gamma$ by deleting
$e$, and ${\Gamma}/e$ is obtained from $\Gamma$ by contracting $e$. (If
$\Gamma$ contains a loop then ${\chi}_{\Gamma}\equiv 0$).  Note: while
discussing the chromatic algebra, we will interchangeably use two variables,
$Q$ and $q$,where $Q=(q+q^{-1})^2$.

The defining contraction-deletion rule (\ref{chromatic poly1}) may be viewed
as a linear relation between the graphs $G, G/e$ and $G\backslash e$, so in
this context it is natural to consider the vector space defined by graphs,
rather than just the set of graphs. Thus consider the set ${\mathcal G}_n$ of
the isotopy classes of planar
graphs $G$ embedded in a rectangle with $n$ endpoints at the top
and $n$ endpoints at the bottom of the rectangle,
and let ${\mathcal F}_n$ denote the
free algebra over ${\mathbb Z}[q,q^{-1}]$ with free additive generators given by
the elements of ${\mathcal G}_n$.  As usual, the multiplication is given by
vertical stacking in the plane.

The local relations among the elements of ${\mathcal G}_n$, analogous to
contraction-deletion rule for the chromatic polynomial, are given in the
figures below.  Note that these relations only apply to {\em inner} edges
which do not connect to the top and the bottom of the rectangle. They are

(\ref{chrom relation 1}) If $e$ is an inner edge of a graph $G$ which is not a loop, then
$G=G/e-G\backslash e$.

(\ref{chrom relation 2}) If $G$ contains an inner edge $e$ which is a loop, then $G=(q^2+1+q^{-2})\;
G\backslash e$.  (In particular, this relation applies if $e$ is a simple
closed curve not connected to the rest of the graph.)
If $G$ contains a $1$-valent vertex (in the interior of the rectangle)
then $G=0$. Graphically:

{\Small
\begin{equation} \label{chrom relation 1} \CPPic{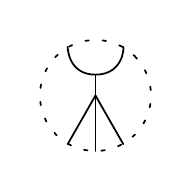} = \CPPic{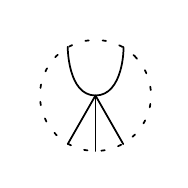} - \CPPic{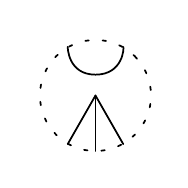},
\end{equation}

\begin{equation} \label{chrom relation 2} \CPPic{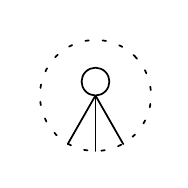}\!\! = (q^2+1+q^{-2}) \CPPic{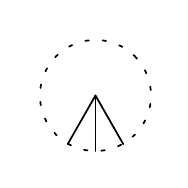} \normaltext{\quad and \quad} \CPPic{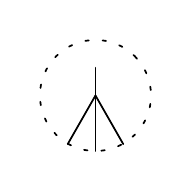} = 0.
\end{equation}}

\begin{definition} \label{chromatic definition} \cite{FK}
The {\em chromatic algebra} in degree $n$, ${\mathcal C}_n$, is an algebra
over ${\mathbb Z}[q]$ which is defined as the quotient of the free graph
algebra ${\mathcal F}_n$ by the ideal $I_n$ generated by the relations
(\ref{chrom relation 1}, \ref{chrom relation 2}) above. Set ${\mathcal
  C}=\cup_n {\mathcal C}_n$.
\end{definition}

The trace, $\tr_{\chi}\co {\mathcal C}\longrightarrow{\mathbb Z}[q]$ is defined
on the additive generators (graphs $G$ in the rectangle $R$) by connecting
the top and bottom endpoints of $G$ by disjoint arcs in complement of $R$
the plane (denote the result by $\overline G$) and evaluating the chromatic
polynomial of the dual graph $\widehat{\overline G}$:

\begin{equation*} \label{chromatic trace}
\tr_{\chi}(G)\; \, =\; \, (q+q^{-1})^{-2}\cdot {\chi}_{\widehat{\overline G}}((q+q^{-1})^2).
\end{equation*}

\subsection{Relations between the diagrammatic algebras} \label{relations}

This section recalls trace-preserving homomorphisms between the $\SO(3)$ BMW, chromatic, and Temperley-Lieb algebras.
A categorified version is given in sections \ref{sobmwcat}, \ref{chromatic cat} below.

\begin{definition} \label{BMW to chrom}
The formulas (introduced in \cite{KV})
$$
\CPic{n2-cross} \mapsto q^{-2} \CPic{n2-1} - \CPic{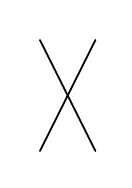} +\; q^{2}
\CPic{n2-s} \normaltext{\quad and \quad} \CPic{n2-cross-2} \mapsto q^{2}
\CPic{n2-1} - \CPic{n2-quad} +\; q^{-2} \CPic{n2-s}
$$
define a homomorphism of algebras $i\co \BMW_n\longrightarrow {\mathcal
  C}_n$ over ${\mathbb Z}[q,q^{-1}]$, see theorem 5.1 in \cite{FK} (see also \cite{FR}).
  \end{definition}

\begin{definition}    \label{def:phi}
Define a homomorphism ${\phi}\co {\mathcal F}_n\longrightarrow \TL_{2n}$ on
the additive generators (graphs in a rectangle) of the free graph algebra
${\mathcal F}_n$ by replacing each edge with the second Jones-Wenzl
projector $P_2$, and
resolving each vertex as shown in the figure below:
$$\CPic{line}\hspace{-.2in} \mapsto \CPic{p2box} = \CPic{n2-1} - \frac{1}{q+q^{-1}} \CPic{n2-s} \textnormal{\quad and \quad} \CPic{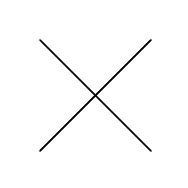} \mapsto (q+q^{-1})\cdot \CPic{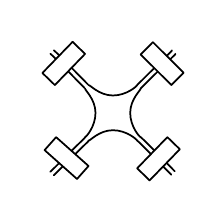}\quad .$$

The factor in the definition of ${\phi}$ corresponding to an $r$-valent
vertex is $(q+q^{-1})^{(r-2)/2}$, so for example it equals $q+q^{-1}$ for the $4$-valent
vertex in the figure above. The overall factor for a graph $G$ is the
product of the factors $(q+q^{-1})^{(r(V)-2)/2}$ over all vertices $V$ of $G$.
\end{definition}

Therefore ${\phi}(G)$ is a sum of $2^{E(G)}$ terms, where $E(G)$ is the
number of edges of $G$. It is shown in lemmas 6.2 and 6.4 in \cite{FK} that
${\phi}$ induces a well-defined homomorphism of algebras ${\mathcal C}_n\longrightarrow \TL_{2n}$. Moreover,
\begin{equation*} \label{chromatic dual}
\tr_\chi(G)\, =\, \tr_{\TL} ({\phi}(G)).\end{equation*} Phrased differently,
up to a renormalization factor $(q+q^{-1})^{-2}$ the chromatic polynomial of
a planar graph may be computed as the Yamada polynomial \cite{Y} of the dual
graph, that is the evaluation of the quantum spin network where each edge is
labeled with the second projector. The following lemma summarizes the above
discussion:

\begin{lemma}  \label{relations lemma}
The homomorphisms $i, {\phi}$ are trace-preserving, in other words the following diagram commutes:
{\Small
\begin{diagram}
\BMW_n & \rTo^i &  \mathcal{C}_n &  \rTo^{\phi} &  \TL_{2n} \\
       & \rdTo  & \dTo & \ldTo  & \\
       &        & \mathbb{Z}[q,q^{-1}] &       & 
\end{diagram}}
\end{lemma}

\section{Categorification of the Temperley-Lieb algebra} \label{categorified TL section}

In this section we recall Dror Bar-Natan's graphical formulation
\cite{DBN} of Khovanov's categorification of the Temperley-Lieb
algebra \cite{KH}.

There is an additive category $\PCob(n)$ whose objects are isotopy classes
of formally $q$-graded Temperley-Lieb diagrams with $2n$ boundary
points. The morphisms are given by the free $\mathbb{Z}$-module spanned by
isotopy classes of orientable cobordisms bounded in $\mathbb{R}^3$ between
two planes containing such diagrams. If $\chi(S)$ is the Euler
characteristic of a surface $S$, then a cobordism $C : q^i A \to q^j B$ has
\emph{degree} given by

$$|C| = \chi(C) - n + j-i .$$

It has become a common notational shorthand to represent a handle by a dot
and a saddle by a flattened diagram containing a dark line:
$$\CPic{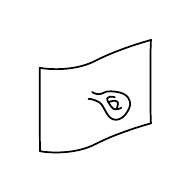} = 2 \CPic{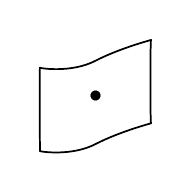} = 2 \CPic{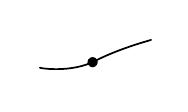} \textnormal{\quad and\quad } \CPic{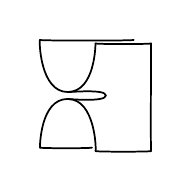} = \CPic{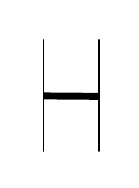}.$$

There are maps from a circle to the empty set and vice versa given by a
punctured sphere and a punctured torus

$$\begin{diagram}
\varphi:\hspace{-.25in} &\CPic{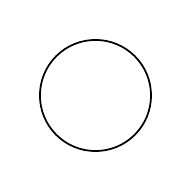}\hspace{.1in} & \pile{\rTo^{\left( \CPic{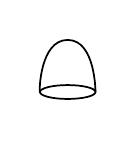} \CPic{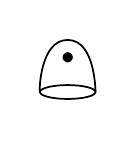} \right)  }\\
\lTo_{\left( \CPic{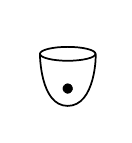} \CPic{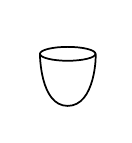} \right)  }}
& \hspace{.1in} q^{-1}\, \emptyset\,\, \oplus\,\, q\, \emptyset &\hspace{.1in} :  \psi & .
\end{diagram}
$$

In order to obtain $\varphi \circ \psi = 1$ and $\psi \circ \varphi = 1$ we
form a new category $\Cob(n) = \Cob^3_{\cdot/l}(n)$ obtained as a quotient
of the category $\PCob(n)$ by the relations given below.
 $$\CPic{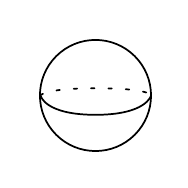} = 0 \hspace{.75in} \CPic{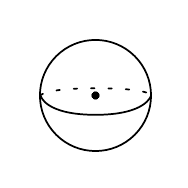} = 1 \hspace{.75in} \CPic{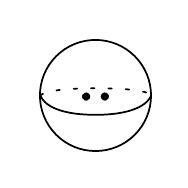} = 0 \hspace{.75in} \CPic{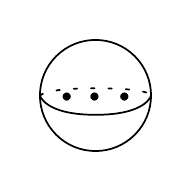} = \alpha$$
 $$\CPic{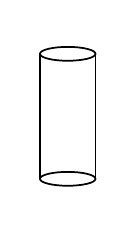} = \CPic{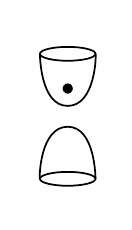} + \CPic{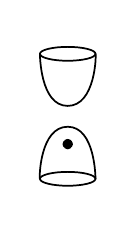}$$

The cylinder or neck cutting relation implies that closed surfaces
$\Sigma_g$ of genus $g > 3$ must evaluate to $0$. In what follows we will
let $\alpha$ be a free variable and absorb it into our base ring ($\Sigma_3=
8\alpha$). One can think of $\alpha$ as a deformation parameter.

In this categorification the skein relation becomes
$$\begin{diagram} \CPic{n2-orcross} \hspace{-.3in} & = & q \CPic{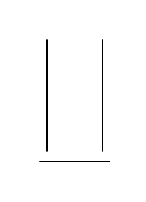}\hspace{-.15in} &\rTo^{\!\MPic{n2-1-sad}} & \hspace{.2in} q^{2} \CPic{n2-s}\hspace{-.15in} & \normaltext{\quad\,\, and } &
\CPic{n2-orcross-2}\hspace{-.3in} & = & q^{-2} \CPic{n2-s} & \rTo^{\!\MPic{n2-1-sad}} & \hspace{.2in} q^{-1}\CPic{n2-1-u}
\end{diagram}
$$

where the underlined diagram represents homological degree 0.

\begin{definition}
  Let $\Kom(n) = \Kom(\Mat(\Cob^3_{\cdot/l}(n)))$ be the category of chain
  complexes of formal direct sums of objects in $\Cob^3_{\cdot/l}(n)$.
\end{definition}

The skein relation allows us to associate to any tangle diagram $D$ with
$2n$ boundary points an object in $\Kom(n)$.

Given two objects $C,D \in \Kom(n)$ we will use $C\otimes D$ to denote the
categorified Temperley-Lieb multiplication $\otimes : \Kom(n) \otimes
\Kom(n) \to \Kom(n)$ obtained by gluing all diagrams and morphisms along
the $n$ boundary points and $n$ boundary intervals respectively.

\subsection{Chain Homotopy Lemmas} \label{homotopy lemmas}

We will make frequent use of the following standard lemma in this paper,

\newcommand{\matot}[2]{\ensuremath{\left(\begin{array}{c} #1 \\ #2 \end{array}\right)}}
\newcommand{\matto}[2]{\ensuremath{\left(\begin{array}{cc} #1 & #2 \end{array}\right)}}
\newcommand{\mattt}[4]{\ensuremath{\left(\begin{array}{cc} #1 & #2\\ #3 & #4 \end{array}\right)}}

\begin{lemma}{(Gaussian Elimination, \cite{DrorFast})} \label{gaussian elimination}
  Let $K_*$ be a chain complex in an additive category $\mathcal A$
  containing a summand of the form given below:

$$\begin{diagram}
A &\rTo^{\matot{\cdot}{\delta}} & B\oplus C & \rTo^{\mattt{\varphi}{\lambda}{\mu}{\eta}} & D\oplus E & \rTo^{\matto{\cdot}{\epsilon}} & F
\end{diagram}$$

Then if $\varphi : B \to D$ is an isomorphism there is a homotopy
equivalence from $K_*$ to a smaller complex containing the summand below
obtained by removing $B$ and $D$ terms via $\varphi$:
$$\begin{diagram}
A &\rTo^\delta & C & \rTo^{\eta - \mu\varphi^{-1}\lambda} & E & \rTo^{\epsilon} & F
\end{diagram}$$
\end{lemma}

\smallskip

The following result is a direct generalization which will be very useful in our context.

\smallskip

\begin{lemma}{(Simultaneous Gaussian Elimination, \cite{CK})} \label{sim gaussian elimination}
Let $K_*$ be a chain complex in an additive category $\mathcal A$ of the form

$$K_* = \begin{diagram} A_0\oplus C_0 & \rTo^{M_0} & A_1\oplus B_1 \oplus C_1 & \rTo^{M_1} &
  A_2 \oplus B_2 \oplus C_2 & \rTo^{M_2} &  \cdots
\end{diagram}$$

where

$$M_0 =
\left(\begin{array}{cc}
a_0 & c_0\\
d_0 & f_0\\
g_0 & j_0
\end{array}\right) \quad \textnormal{   and   } \quad M_i =
\left(\begin{array}{ccc}
a_i & b_i & c_i\\
d_i & e_i & f_i\\
g_i & h_i & j_i
\end{array}\right) \textnormal{ for all $i > 0$ }
$$

If $a_{2i} : A_{2i} \to A_{2i+1}$ and $e_{2i+1} : B_{2i+1} \to B_{2i+2}$ are
isomorphisms for $i \geq 0$ then the chain complex $K_*$ is homotopy
equivalent to the smaller chain complex $D_*$ obtained by removing all $A_i$
and $B_i$ terms via the isomorphisms $a_{2i}$ and $e_{2i+1}$:

$$D_* = \begin{diagram}
C_0 & \rTo^{q_0} & C_1 & \rTo^{q_1} & C_2 &\rTo^{q_2} & C_3 & \rTo^{q_3} & \cdots
\end{diagram}
$$

where $q_{2i} = j_{2i} - g_{2i} a_{2i}^{-1} c_{2i}$ and $q_{2i+1} = j_{2i+1} - h_{2i+1} e_{2i+1}^{-1} f_{2i+1}$.

\end{lemma}

\section{Construction of the second projector} \label{formulasec}

In this section we define a chain complex $P_2 \in \Kom(2)$ which
categorifies the second Jones-Wenzl projector (definition \ref{jwproj
  tl}). This construction of $P_2$ is universal and unique up to homotopy
\cite{CK}. (Other definitions were obtained in
\cite{FSS} and \cite{Rozansky}).

\subsection{The Second Projector Revisited} \label{second projector section}

The second projector is defined to be the chain complex
$$\begin{diagram}
\CPic{p2box}\! & = & \!\CPic{n2-1} & \rTo^{\MPic{n2-1-sad}} &q \CPic{n2-s} &\rTo^{\MPic{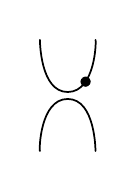} \!\!\!- \MPic{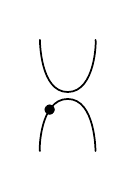}} &q^{3} \CPic{n2-s} &
\rTo^{\MPic{n2-tops}\!\!\! + \MPic{n2-bots}} & q^{5} \CPic{n2-s} &  \cdots
\end{diagram}$$

in which the last two maps alternate ad infinitum. More explicitly,
$$P_2 = (C_*, d_*), $$

the chain groups are given by
$$
C_n =
\left\{
\begin{array}{lr}
q^0 \MPic{n2-1}&n = 0\\
q^{2n-1} \MPic{n2-s}&n > 0
\end{array}
\right .
$$

and the differential is given by

$$
d_n =
\left\{
\begin{array}{llr}
\MPic{n2-1-sad}                 & : \MPic{n2-1}\!\! \to q \MPic{n2-s}   & n = 0\\
\MPic{n2-tops} \!\!\! + \MPic{n2-bots} & :   q^{4k-1} \MPic{n2-s}\!\!  \to q^{4k +1} \MPic{n2-s}                    &n \ne 0, n = 2k\\
\MPic{n2-tops} \!\!\! - \MPic{n2-bots} & : q^{4k +1} \MPic{n2-s}\!\!  \to q^{4k+3} \MPic{n2-s}                    & n = 2k+1\,.
\end{array}
\right .
$$

\begin{proposition}
$P_2$ defined above is a chain complex, that is successive compositions of the differential are equal to zero.
\end{proposition}
\begin{proof}
Since $d_{2n+1} \circ d_{2n} = d_{2n} \circ d_{2n-1}$ there are only two cases:
\begin{eqnarray*}
d_1\circ d_0 &= & \MPic{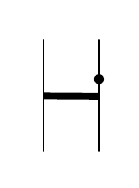}\!\!\! - \MPic{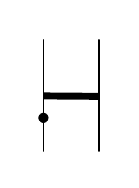} \\
&= & \MPic{n2-s-bots} \!\!\!- \MPic{n2-s-bots} =  0\\
& \normaltext{ and } & \\
 d_{2n+1} \circ d_{2n} &= & (\MPic{n2-tops} \!\!\! + \MPic{n2-bots})\circ (\MPic{n2-tops}\!\!\! - \MPic{n2-bots}) \\
&=& \MPic{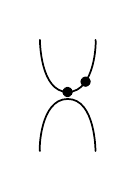} \!\!\! + \MPic{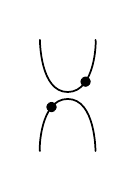}\!\!\! - \MPic{n2-mid-2} \!\!\! - \MPic{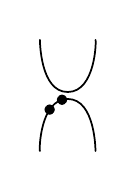} \\
&=& \alpha \MPic{n2-s} \!\!\! + 0 - \alpha \MPic{n2-s} = 0 .\\
\end{eqnarray*}
\end{proof}

\begin{theorem}{(\cite{CK})}\label{secondproj thm}
  The chain complex $P_2 \in \Kom(2)$ defined above is contractible ``under
  turnback'' and a homotopy idempotent. Graphically,
$$\CPPic{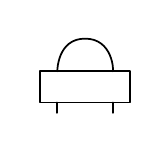} \simeq 0 \textnormal{\quad\quad and\quad\quad } \CPPic{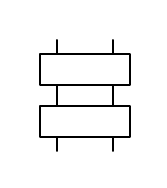} \simeq \CPPic{p2box}. $$

Algebraically, these are the relations

$$P_2 \otimes e_1 \simeq 0 \simeq e_1 \otimes P_2 \normaltext{\quad and \quad} P_2 \otimes P_2 \simeq P_2. $$

\end{theorem}
\begin{proof}
We will prove the turnback property first. Note that the vertical symmetry
in the definition of $P_2$ implies $P_2\otimes e_1 \cong e_1\otimes
P_2$. Consider $e_1 \otimes P_2$:

$$\begin{diagram}
\CPic{p2capped} \!\!\! & = & \!\!\! \CPic{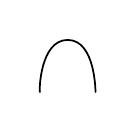} &
\rTo^{\MPic{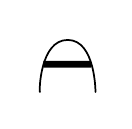}} &
q \CPic{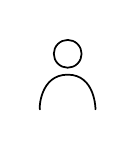} &
\rTo^{\MPic{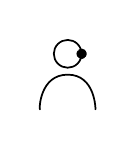} \!\!\!- \MPic{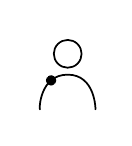}} &
q^{3} \CPic{n2-capbub} &
\rTo^{\MPic{n2-capbub-tops} \!\!\! + \MPic{n2-capbub-bots}} &
q^{5} \CPic{n2-capbub} &  \!\!\! \!\!\!\cdots .
\end{diagram}$$

We ``deloop'' and conjugate our differentials by the isomorphism $\varphi$
in section \ref{categorified TL section} to obtain the isomorphic complex
$$\begin{diagram}
\CPic{n2-hs} \!\!\! &
\rTo^{\!\!\! A} &
q^{0} \CPic{n2-hs} \!\!\! \oplus\,\, q^{2} \CPic{n2-hs} \!\!\!&
\rTo^{B} &
q^{2} \CPic{n2-hs} \!\!\! \oplus\,\, q^{4} \CPic{n2-hs} \!\!\!&
\rTo^{C} &
q^{4} \CPic{n2-hs} \!\!\! \oplus\,\, q^{6} \CPic{n2-hs}\!\!\! &  \cdots
\end{diagram}$$

where $A = \left(\MPic{n2-hs} \MPic{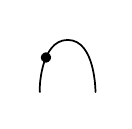} \!\!\!\right)$,

\begin{align*}
B &= \left(\begin{array}{cc}
-\MPic{n2-hs-dot} & \MPic{n2-hs}\\
\alpha \MPic{n2-hs} & -\MPic{n2-hs-dot}
\end{array}
\right) & \normaltext{and\quad\quad} &
C = \left(\begin{array}{cc}
\MPic{n2-hs-dot} & \MPic{n2-hs}\\
\alpha \MPic{n2-hs} & \MPic{n2-hs-dot}
\end{array}
\right). & \\
\end{align*}

Applying lemma \ref{sim gaussian elimination} (simultaneous Gaussian
elimination) by using the identity map in the first component of the first
map and the identity in the upper righthand component of each successive
matrix shows that the complex is homotopic to the zero complex.

The relation $P_2 \otimes P_2 \simeq P_2$ follows from expanding either the
top or bottom projector and again using lemma \ref{sim gaussian elimination}
to contract all of the projectors containing turnbacks as above. What
remains is the chain complex for $P_2$ in degree 0.
\end{proof}

\section{Categorification of the $\SO(3)$ BMW algebra} \label{sobmwcat}

In this section we show that the chain complexes obtained by applying the
second projector to the strands of a 2-cabling are invariant under
Reidemeister moves and satisfy relations categorifying those of the $\SO(3)$
BMW algebra.

As in section \ref{so bmw sec}, to any diagram $D$ associate a chain complex
$F(D)$ in the category $\Kom(2n)$ by replacing each strand in $D$ with two
parallel strands composed with the second projector. (Note that using the
categorified Kauffman skein relation in section \ref{categorified TL
  section} one associates a chain complex to oriented tangles and the two
parallel strands in the current construction are given opposite
orientations).  This can be illustrated by
$$\CPPic{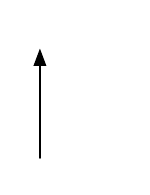}\hspace{-.4in}\mapsto \CPPic{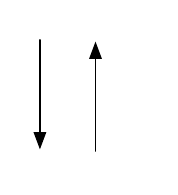}\hspace{.7in} \CPPic{line}\hspace{-.4in} \mapsto \CPPic{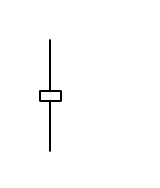} \hspace{.7in} \CPPic{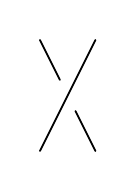}\hspace{-.2in} \mapsto \CPPic{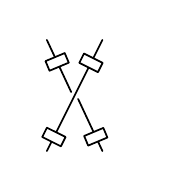}$$

Formally, this construction categorifies the $2$-colored Jones polynomial,
see \cite{CK} and section \ref{2 colored} for further discussion.  In the
remainder of this section we prove that the Reidemeister moves and $\SO(3)$
skein relation are satisfied up to homotopy.

\begin{lemma}{(Projector Isotopy)} \label{projector isotopy}
  A free strand can be moved over or under a projector up to homotopy. In
  pictures,
$$\CPPic{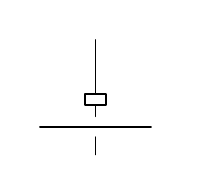} \simeq \CPPic{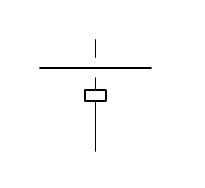}\hspace{1in}\CPPic{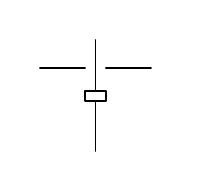} \simeq \CPPic{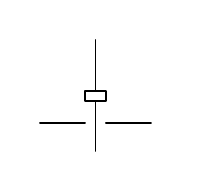}$$
\end{lemma}

\begin{proof} The chain complex for the diagram with the projector below the
  strand and the chain complex for the diagram with the projector above the
  strand are chain homotopy equivalent to the chain complex $C$ for the
  diagram with {\em two} projectors: one above the strand and one below the
  strand. This is true because expanding either of the two projectors in $C$
  gives the identity diagram in degree zero and every other term involves a
  turnback, which is contractible when combined with the second copy of the
  projector.
\end{proof}

This lemma allows us to show that the Reidemeister moves are satisfied.

\begin{theorem}
  This construction yields invariants of framed tangles.
\end{theorem}

\begin{proof}
For the second Reidemeister move,

$$\CPPic{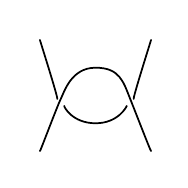} = \CPPic{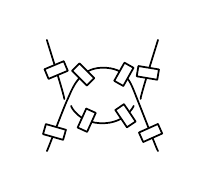} \simeq \CPPic{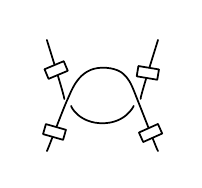} \simeq \CPPic{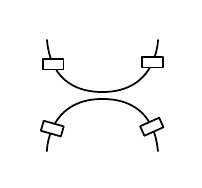} \simeq \CPPic{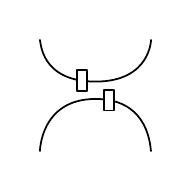} = \CPPic{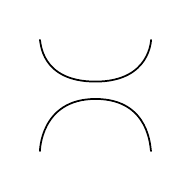}$$

The first equality is by definition. The homotopy equivalence follows from
the projector isotopy lemma and $P_2\otimes P_2 \simeq P_2$. We then apply
the original second Reidemeister move and $P_2\otimes P_2 \simeq P_2$
again. The argument for the third Reidemeister move features the same ideas.

$$\CPPic{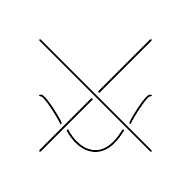} = \CPPic{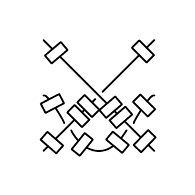} \simeq \CPPic{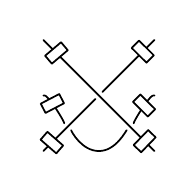} \simeq \CPPic{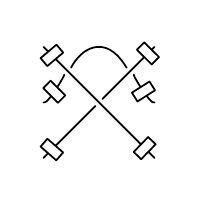} \simeq \CPPic{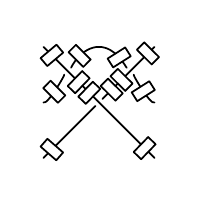} = \CPPic{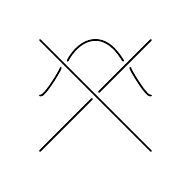} $$

The $q^{\pm 4}$-skewed version of first Reidemeister move (section \ref{so
  bmw sec}) are satisfied by our construction.
$$\CPPic{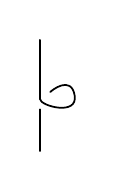}\hspace{-.2in} = \CPPic{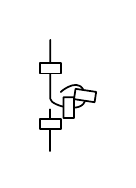}\hspace{-.2in} \simeq  \CPPic{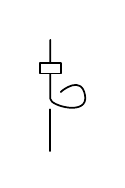}$$

and
$$\CPPic{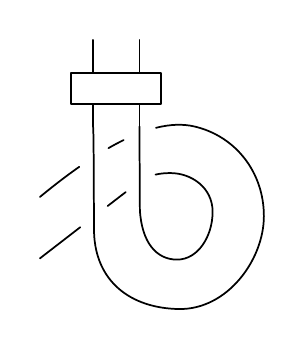}\hspace{.5in} \simeq\, t^{2} q^{4} \CPPic{p2box}.$$

Where $t^{2} q^{4}$ denotes bidegree $(2,4)$. This is obtained by
expanding all of the crossings, delooping and contracting the remaining
subcomplex consisting of projectors containing turnbacks. We've shown

$$\CPic{flipper2} \simeq\, q^{2(N-1)}\CPic{line}$$

with $N=3$. The opposite crossing follows from the same argument.
\end{proof}

\subsection{${\mathbf{SO(3)}}$ BMW Skein Relation} \label{SO3 section}

In order to prove that the first skein relation pictured in section \ref{so bmw sec}
is satisfied by our
categorification we consider the chain complex associated to a
crossing:
\begin{equation}\label{crossing eq}
\CPPic{orline}\hspace{-.4in}\mapsto \CPPic{doubleor}\hspace{1in}\CPPic{n2-unorcross} \mapsto \CPPic{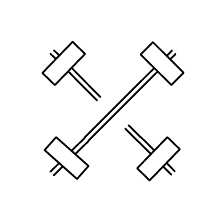}
\end{equation}

Now expanding all four crossings on the right hand side yields a chain
complex with 16 terms. (The reader may find it helpful to draw the diagram with all
16 terms to follow the argument below.) We will use the convention below to index
resolutions:
$$(abcd) = \BPic{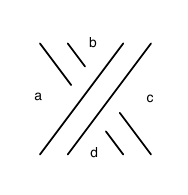} \normaltext{\quad where \quad} \begin{diagram} \CPPic{n2-1} & \lTo^0 & \CPPic{n2-unorcross} & \rTo^1 & \CPPic{n2-s} \end{diagram}$$

There is one circle corresponding to the (0101) resolution which can be
delooped and Gaussian elimination can be performed removing the terms
corresponding to the (0001) resolution and the (1101) resolution. Nine
of the remaining terms contain projectors with turnbacks.\footnote{Those
  corresponding to (1000), (0010), (1100), (1010), (1001),
  (0110), (0011), (1110) and (1011) resolutions.}  Contracting
using lemma \ref{sim gaussian elimination} these yields the chain complex
$$\begin{diagram} t^{-2} q^{-2} \CPPic{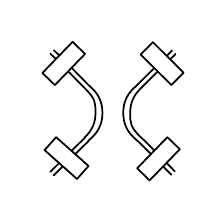} & \rTo & t^{-1} q^{-1} \CPPic{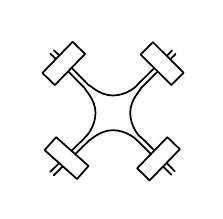} & \rTo & tq \CPPic{case1} & \rTo & t^2 q^2 \CPPic{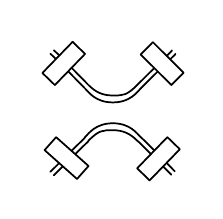}
\end{diagram}
$$

giving a categorification of the crossing formula in definition \ref{BMW to
  chrom}. The factor $(q+q^{-1})$ which comes from the two terms in the
middle is seen in the translation of the $4$-valent graph to the
Temperley-Lieb algebra (see definition \ref{def:phi} of the homomorphism
$\phi$.) Note that the diagram above is only a schematic illustration of the
chain complex for the resolution of the crossing at the beginning of section
\ref{SO3 section}: the contractions mentioned above produce maps which are
not illustrated in the above diagram. Next we will examine this chain
complex in more detail.

We now proceed to show that the relation
\begin{equation}\label{another Kauffman}
\CPic{n2-cross} - \CPic{n2-cross-2} = (q^2 - q^{-2})\Bigg(\CPic{n2-1} - \CPic{n2-s}\Bigg)
\end{equation}

holds in our category, this requires a more detailed analysis of the chain
complex considered above.  Begin by again expanding all four crossings in
(\ref{crossing eq}), corresponding to the the leftmost term in the equation
above. We obtain a chain complex with 16 terms with one term in homological
degrees $-2$ and $2$, four terms in degrees $-1$ and $1$ and six terms in
degree $0$. Form a new chain complex,
$$\begin{diagram}
 & & & &  & & 0 &\rTo & \hspace{-.1in}\CPic{case2}\hspace{.1in}\\
 &  & &       &  & & \dTo & & \dTo^1 \\
\left[\hspace{-.1in}\CPic{so-cross}\hspace{.1in}\right]_{-2}  & \rTo^{d_{-2}} & \left[\hspace{-.1in}\CPic{so-cross}\hspace{.1in}\right]_{-1} & \rTo^{d_{-1}} &\left[\hspace{-.1in}\CPic{so-cross}\hspace{.1in}\right]_0 & \rTo^{d_0} & \left[\hspace{-.1in}\CPic{so-cross}\hspace{.1in}\right]_1 & \rTo^{d_1} & \left[\hspace{-.1in}\CPic{so-cross}\hspace{.1in}\right]_2 \\
\dTo^{1} &         & \dTo & & & & & &\\
\hspace{-.1in}\CPic{case0}\hspace{.1in} &\rTo & 0 & & & & & &
\end{diagram}$$

The graded Euler characteristic of this complex is the quadrivalent vertex
in definition \ref{def:phi}. Contracting the first and last maps using the
introduced isomorphisms yields the chain complex
$$\begin{diagram} \left[\hspace{-.1in}\CPic{so-cross}\hspace{.1in}\right]_{-1} & \rTo^{d_{-1}} & \left[\hspace{-.1in}\CPic{so-cross}\hspace{.1in}\right]_0  & \rTo^{d_0} & \left[\hspace{-.1in}\CPic{so-cross}\hspace{.1in}\right]_{-1}
\end{diagram}$$

The maps $d_{-1}$ and $d_0$ remain the same as in the previous diagram and
so consist of saddles between resolutions of crossings. Now contract terms
in degrees $-1$ and $1$ that are diagrams with projectors capped by
turnbacks\footnote{Terms corresponding to (1000), (0010), (1110)
  and (1011) resolutions.}. Observe again that contracting these will not
affect the maps between remaining terms. There remains a contractible term
(1010) in degree zero (with four turnbacks) which is a direct summand of the
chain complex, that is there are no arrows starting or ending at this term,
so that contracting this term does not affect the maps between the remaining
terms.  Again delooping the term in the center corresponding to the
(0101) resolution allows one to cancel terms corresponding to (0001)
and (0111) resolutions in degrees $-1$ and $1$ respectively. These
cancelations in fact do change the maps between the remaining terms, the
resulting maps can be analyzed using the Gaussian elimination lemma
\ref{gaussian elimination}, and the result is given below.  The chain
complex
$$\begin{diagram}
\CPPic{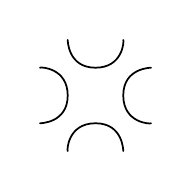} & \rTo &  \CPPic{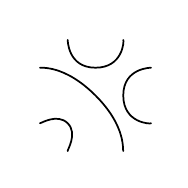} \oplus \CPPic{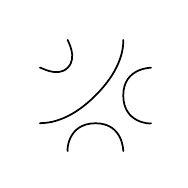} \oplus \CPPic{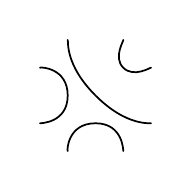} \oplus \CPPic{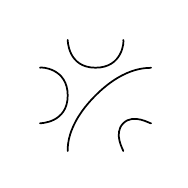}  & \rTo & \CPPic{harpo1}\\
\end{diagram}$$

is what remains. All of the maps are saddles. Note that
all of the diagrams contain four projectors which are not pictured. The
first and last terms are the chain complex associated to the planar crossing
(the middle term in the equality below), while the four terms in the middle
have a projector capped with a turnback, and are therefore contractible.

On the other hand, expanding the lefthanded crossing in (\ref{another
  Kauffman}) rather than the righthanded one and carrying out the same
argument yields precisely the same complex! This is clear since the terms in
the diagram above are ${\pi}/2$ rotationally symmetric. It follows that in
the image of the Grothendieck group,

$$q^2\CPic{n2-1} - \CPic{n2-cross} + q^{-2}\CPic{n2-s} = \CPic{n2-quad} =  q^{-2}\CPic{n2-1} - \CPic{n2-cross-2} + q^{2}\CPic{n2-s}$$

which is equivalent to the desired relation (\ref{another Kauffman}). \qed

\subsection{Ribbon graphs} \label{ribbon graphs section}
A {\em ribbon graph} is a pair $(G,S)$ where $G$ is a graph embedded in a
surface $S$ with boundary, and the inclusion $G\subset S$ is a homotopy
equivalence. Our construction gives an invariant of ribbon graphs embedded
in the $3$-sphere. Specifically, to a ribbon graph $(G,S)$ associate a chain
complex as follows: Replace each edge of $G$ with the second Jones-Wenzl
projector $P_2$, and using the ribbon structure resolve each vertex as in
the figure below:
\begin{equation} \label{ribbon figure}
\CPic{line}\hspace{-.2in} \mapsto \CPic{p2box} \textnormal{\quad and
  \quad} \CPic{fourvalent} \mapsto \CPic{fourvalentexp}\quad .
  \end{equation}
  The
resulting curves in the neighborhood of each vertex are oriented as the
boundary of a regular neighborhood of the graph $G$ in $S$.

It is an interesting question to determine how powerful this invariant is,
and in particular whether this homology theory may be used to detect planar
graphs. Given a connected ribbon graph $(G,S)$ embedded in $S^3$,
contracting a maximal tree gives a map to the graph $G'$ with a single
vertex and a number of loops (with the same underlying surface, embedded in
$S^3$). There is an induced map on chain complexes (which amounts to the
projection onto the homological degree zero for each contracted edge, see
section \ref{delcontrcat} below.)  If the embedding of $(G,S)$ into $S^3$ is
isotopic to a planar embedding, then the homology of $G'$ is the chromatic
homology of a tree, computed in the Appendix. Analyzing the homology of
planar graphs motivated the following conjecture.

\begin{conjecture} A ribbon graph $(G,S)$ embedded into $S^3$ is isotopic to a planar graph if and only if
its homology groups $H_i$ are trivial for $i<0$, and $H_0$ is free of rank $2$.
\end{conjecture}

A related question is to determine whether the genus of the ribbon graph (defined as the genus of the underlying
surface $S$) is determined by this homology theory.

\section{Chromatic Categorification} \label{chromatic cat}

In this section we show that our construction produces a categorification of
the chromatic polynomial of planar graphs. To each planar graph $G$ we associate
a chain complex $\inp{G}$ whose graded Euler characteristic is a particular normalization of the chromatic polynomial.

Our construction differs in significant ways from other categorifications of
the chromatic polynomial present in the literature \cite{GR,Stosic}. In particular,
it depends on a specific choice of Frobenius algebra. This follows from the
relations in section \ref{categorified TL section}. While this rigidity may
have the disadvantage of limiting the variety of answers that our theory
provides, it allows for an extension to invariants of ribbon graphs embedded
in $\mathbb{R}^3$. This information then enriches the structure of the
underlying chromatic polynomial. See section \ref{ribbon graphs section} for more details.

In section \ref{delcontrcat} below we show that to each edge $e \in G$ which
is not a loop there is a contraction-deletion long exact sequence on the
homology of $\widehat{G}$ corresponding to the contraction-deletion relation
of section \ref{chromaticpolydef}.

\subsection{A categorification of the chromatic polynomial} \label{howdy}

In order to associate to a planar graph $G$ a chain complex $\inp{G}$
with the correct Euler characteristic, we define $\inp{G}$ to be
the evaluation of the dual graph $\widehat{G}$ in the $\SO(3)$ BMW
categorification of section \ref{sobmwcat}. For example,
{\Small
$$\hspace{-.8in}\BPic{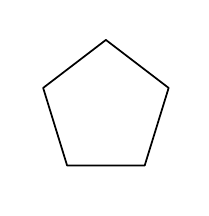} \hspace{.4cm} \mapsto \hspace{.4cm} \begin{minipage}{.4in}
\includegraphics[scale=.3]{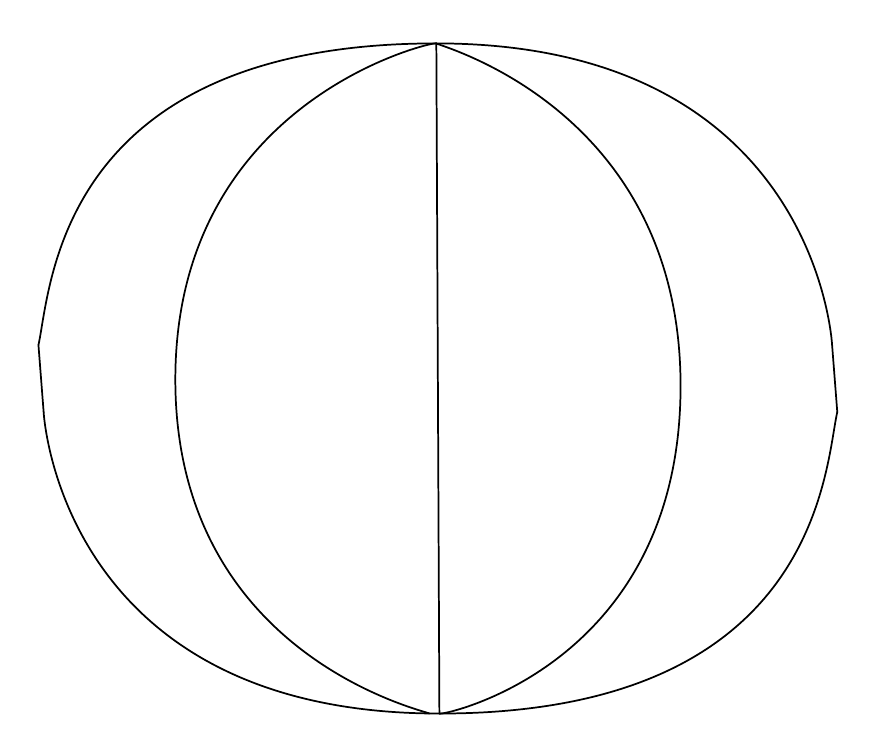}
\end{minipage} \hspace{.6in} \hspace{.4cm} \mapsto \hspace{.4cm} \begin{minipage}{.4in}
\includegraphics[scale=.325]{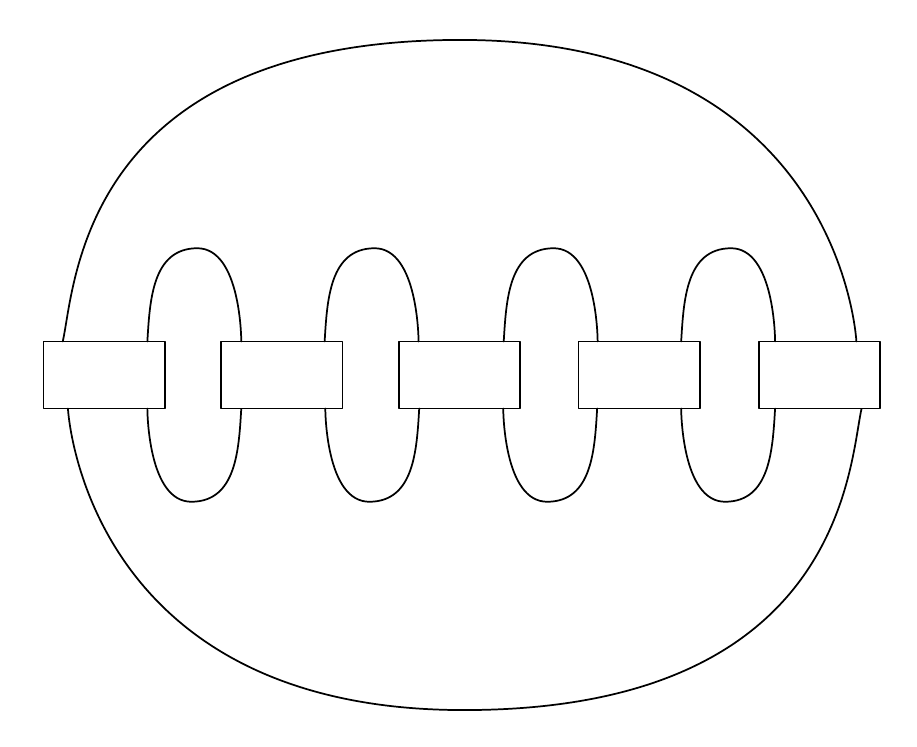}
\end{minipage}$$
}

The pentagon is dual to the graph $\theta_5$ to its right. Associated to $\theta_5$ is
the chain complex given in section \ref{ribbon graphs section}: replace each edge with
a pair or parallel strands with the second Jones-Wenzl projector, and connect the strands near each
vertex to get a planar diagram. (The
homology of $\theta_n$ for all $n > 1$ is given in the appendix).

Defining maps between planar graphs $G$ and $H$ to be chain maps between the
associated chain complexes $\inp{G}$ and $\inp{H}$ yields a category
$\mathcal{C}'$.

\begin{theorem} \label{chromatic theorem}
The category $\mathcal{C}'$ categorifies the chromatic algebra
$\mathcal{C}_0$. In particular, if $G$ is a planar graph then up to a
normalization the graded Euler characteristic of $\inp{G}$ is the chromatic
polynomial $\chi_{G}$ evaluated at $(q+q^{-1})^2$:
$$\chi_{G}\left((q+q^{-1})^2\right) = (q+q^{-1})^2 \prod_v (q+q^{-1})^{(r(v)-2)/2}\; \chi_q\inp{G},$$
where the product is taken over all vertices $v$ of the dual graph $\widehat G$ and $r(v)$ is the valence of $v$ (see definition \ref{def:phi}).
The above equation holds in the ring of formal power series
$\mathbb{Z}\llbracket q \rrbracket$.
\end{theorem}

The proof of this theorem follows immediately from the discussion in section \ref{chromaticpolydef} and lemma \ref{relations lemma}
together with section \ref{sobmwcat}. (To
be precise, $\langle G\rangle$ is a categorification of the Yamada
polynomial of the dual graph \cite{Y} which is defined as the evaluation of
the spin network where each edge is labeled with the second projector).

\subsection{The contraction-deletion rule} \label{delcontrcat}

The chain complex $\inp{G}$ associated to a planar graph $G$ in section
\ref{howdy} above satisfies a version of the contraction-deletion rule. For
any edge $e\in G$ which is not a loop there is an exact triangle
\begin{equation}
\label{exact triangle eq}
\begin{diagram}  [\inp{G/e}] & \rTo & \inp{G} & \rTo & \inp{G\backslash e} \end{diagram}
\end{equation}

in the category $\mathcal{C}'$, where $[\inp{G/e}]$ is a certain chain
complex associated to $G/e$ which may be interpreted as $(q+q^{-1})^{-1}\inp{G/e}$.
There is a functor $F$ from the category $\mathcal{C}'$ to abelian
groups, given by associating to each circle a Frobenius algebra
\cite{DBN}. The homology groups of chain complexes fitting into any exact
triangle form a long exact sequence in the image of $F$ (\cite{W} 10.1.4
p.372).

Let $e$ be an edge (not a loop) of a planar graph $G$.  Consider the edge of
the dual graph, intersecting the edge $e$ in a single point. The
construction sends this dual edge to two parallel lines with a projector as in the figure on the left in (\ref{ribbon figure}).
By definition (section \ref{second projector section})
this projector is expanded into the chain complex

{\Small
$$\begin{diagram}
\CPPic{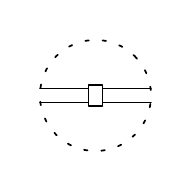} & = &  \CPPic{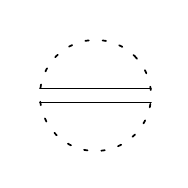} & \rTo^{\MPic{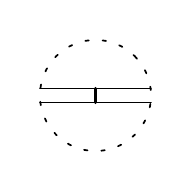}\hspace{.1in}} & \Bigg[ \hspace{-.06in} \CPPic{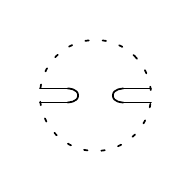}\hspace{-.05in} \Bigg] \quad & =  &\quad \Cone\Bigg(\hspace{-.06in}\CPPic{horzarcsad}\hspace{-.05in}\Bigg)
\end{diagram}$$ }

where all of the terms besides the first one have been collected into the
chain complex with brackets on the right hand side above. This gives an
exact triangle by definition of the $\Cone$ complex (\cite{W} p.18,
p.371). Dualizing again yields the exact triangle (\ref{exact triangle eq}).

Note that on the level of the graded Euler characteristic (\ref{exact triangle eq}) corresponds to a re-normalized
version of the contraction-deletion rule: the term ${\chi}_{{\Gamma}/e}$ in (\ref{chromatic poly1}) acquires a coefficient
$(q+q^{-1})^{-1}$. This version of the contraction-deletion rule corresponds to the re-normalized chromatic polynomial
discussed in theorem \ref{chromatic theorem}.

\section{Computations} \label{computations}

\subsection{Homology of the unknot} \label{homology subsection}

The chain complex associated to the unknot is the ``Markov trace'' of the second
projector $P_2$ (section \ref{second projector section}). The trace of the second projector $p_2 \in \TL_2$ is given by
$$\cTrace\;= [3] = q^{-2} + 1 + q^2 .$$

Our categorification has this graded Euler characteristic when the perimeter
$\alpha = 0$. It is however \emph{not true} that the homology of $\tr(P_2)$
is spanned \emph{only} by classes that correspond to coefficients of the
graded Euler characteristic; the homology contains infinitely many terms
which cancel in the graded Euler characteristic. For further discussion see
\cite{CK}.

Taking the trace of our projector yields a complex with alternating
differential:
$$\begin{diagram}
\CPic{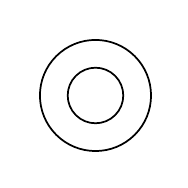} &
\rTo^{\hspace{.2in}\MPic{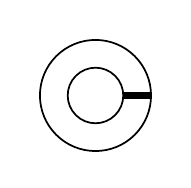}\hspace{.2in}} &
q \CPic{circle} &
\rTo^{\hspace{.2in}0\hspace{.2in}} &
q^{3} \CPic{circle} &
\rTo^{\hspace{.2in}2 \MPic{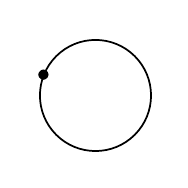}\hspace{.2in}} &
q^{5} \CPic{circle} & \rTo^{\hspace{.2in}0\hspace{.2in}} & \cdots
\end{diagram}$$

Recall that $8\alpha = \Sigma_3$. The homology of this complex is given by

$$
\Homology_n(\tr(P_2)) =
\left\{
\begin{array}{lr}
q^{-2}\mathbb{Z} \oplus q^0\mathbb{Z} & n = 0, \alpha = 0\textnormal{ or } \alpha \ne 0\\
0                                   & n = 1, \alpha = 0\textnormal{ or } \alpha \ne 0\\
q^{4k-2}\mathbb{Z}                    &n = 2k, \alpha  = 0\\
q^{4k+2}\mathbb{Z} \oplus q^{4k}\mathbb{Z}/2      & n = 2k+1, \alpha = 0\\
0                                   & n = 2k, \alpha \ne 0\\
q^{4k+2}\mathbb{Z}/(2\alpha) \oplus q^{4k}\mathbb{Z}/2      & n = 2k+1, \alpha \ne 0\\
\end{array}
\right.
$$

\subsection{Homology of the theta graph} \label{theta homology}
We begin by expanding the middle projector,

{\Small \[
\begin{diagram}
\CPic{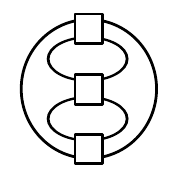}& = &\CPic{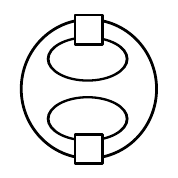}&\rTo^{\MPic{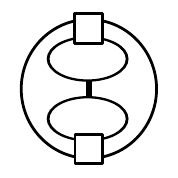}} &q\Cpic{theta_1xx}&\rTo^{\MPic{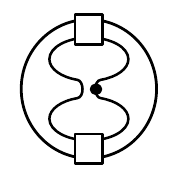}-\MPic{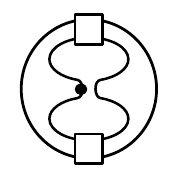}} &q^3\Cpic{theta_1xx}&\rTo^{\MPic{theta_1xx-rightdot}+\MPic{theta_1xx-leftdot}} &q^5\Cpic{theta_1xx}&\rTo& \cdots
\end{diagram}
\]}

If we then expand the top projector,

{\Small
\[
\begin{diagram}
    \CPic{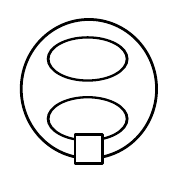}    &\rTo^{\MPic{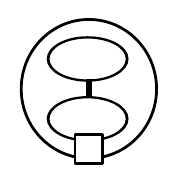}}   &q\Cpic{theta_10x}   &\rTo^0   &q^3\Cpic{theta_10x}  &\rTo^{2\MPic{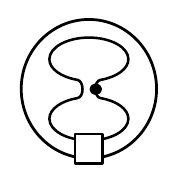}}   &q^5\Cpic{theta_10x}  &\rTo^0 &\cdots  \\
    \dTo^{\MPic{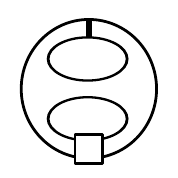}}                &       &\dTo                &       &\dTo                 &       &\dTo&\\
    q\CPic{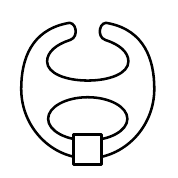}   &\rTo   &\cdot &\rTo   &\cdot  &\rTo   &\cdot  &\rTo &\cdots   \\
    \dTo^0                &       &\dTo                &       &\dTo                 &       &\dTo&\\
    q^3\Cpic{theta_01x} &\rTo   &\cdot &\rTo   &\cdot  &\rTo   &\cdot  &\rTo &\cdots   \\
    \dTo^{2\Mpic{theta_01x-rightdot}}                &       &\dTo                &       &\dTo                 &       &\dTo&\\
    q^5\Cpic{theta_01x} &\rTo   &\cdot &\rTo   &\cdot  &\rTo   &\cdot  &\rTo  &\cdots   \\
    \dTo^0                &       &\dTo                &       &\dTo                 &       &\dTo &\\
    \vdots              &       &\vdots              &       &\vdots               &       &\vdots
\end{diagram}
\]}

The middle terms are all projectors containing turnbacks, which form a contractible subcomplex. Contracting these
yields a homotopy equivalent complex which is a direct sum of

\begin{enumerate}
\item
{\Small
\[
\begin{diagram}
\CPic{theta_00x} & \rTo^{\left(\begin{array}{cc} \MPic{theta_00x-saddlebottom} \\ \MPic{theta_00x-saddletop} \end{array} \right)}    &  q \CPic{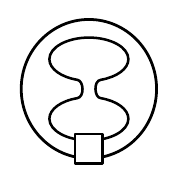} \oplus  q \CPic{theta_01x}     \\
\end{diagram}
\]}
\item
{\Small \begin{diagram}
\bigoplus_k  q^{4k-1} \CPic{theta_10x}& \rTo^{2\MPic{theta_10x-rightdot}} & q^{4k+1} \CPic{theta_10x} \normaltext{\quad and \quad} \bigoplus_k q^{4k-1} \CPic{theta_01x} & \rTo^{2\MPic{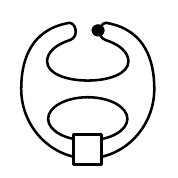}} & q^{4k+1} \CPic{theta_01x}
\end{diagram}
}

Expanding the projector in either case shows that these are isomorphic chain
complexes. Let's \emph{define} $E$ to be this chain complex.

\end{enumerate}

In (1) the circle can be delooped yielding

{\Small
\[
\begin{diagram}
q^{-1} E & = \quad & q^{-1} \cTrace\quad  & \rTo & q \cTrace \\
\end{diagram}
\]}

after a Gaussian elimination. We are left with the task of computing $E$. We have

{\Small
\[
\begin{diagram}
\CPic{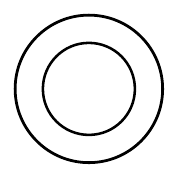}\quad & \rTo^{\MPic{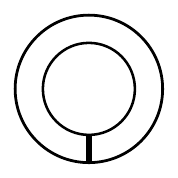}} & q^2 \CPic{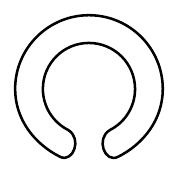}\quad & \rTo^{0} & q^3 \CPic{trace_1}\quad & \rTo^{2 \MPic{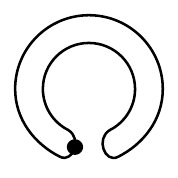}} & q^5 \CPic{trace_1}\quad & \rTo^{0} &\cdots \\
\dTo^{2 \MPic{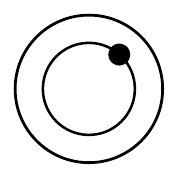}}           &      & \dTo^{2 \MPic{trace-leftdot}}           &      & \dTo^{2 \MPic{trace-leftdot}}           &      & \dTo^{2 \MPic{trace-leftdot}}           &      & \\
q^2 \CPic{trace_0}\quad & \rTo^{-\MPic{trace-saddle}} & q^3 \CPic{trace_1}\quad & \rTo^{0} & q^5 \CPic{trace_1}\quad & \rTo^{-2 \MPic{trace-leftdot}} & q^7 \CPic{trace_1}\quad & \rTo^{0} &\cdots .
\end{diagram}
\]}

The second column can be removed by delooping leaving a sum of chain
complexes of the form
$$
\begin{array}{ccc}
\begin{diagram}
q^{-1} \CPic{circle}\quad \\\dTo^{2 \MPic{circle-dot}} \\ q\CPic{circle}\quad
\end{diagram}
& \textnormal{ and }    & \begin{diagram}
q^0 \CPic{circle}\quad & \rTo^{2 \MPic{circle-dot}} & q^2 \CPic{circle}\\
 \dTo_{2 \MPic{circle-dot}} & & \dTo_{2 \MPic{circle-dot}} \\
q^2 \CPic{circle}\quad & \rTo^{-2 \MPic{circle-dot}} & q^4 \CPic{circle} \, .\\
\end{diagram}
\end{array}
$$

The first complex appears once at the origin of $E$, it has homology $q^{-2}
\mathbb{Z}$ in degree $0$ and $q^0 \mathbb{Z}/2 \oplus q^2 \mathbb{Z}$ in
degree $1$ when $\alpha = 0$. The second appears countably many times, it
has homology $q^{-1}\mathbb{Z}$ in degree $0$, $q \mathbb{Z} \oplus q
\mathbb{Z}/2 \oplus q^3 \mathbb{Z} $ in degree $1$ and $q^3 \mathbb{Z}/2
\oplus q^5 \mathbb{Z}$ in degree $2$. This can be summarized as follows,
 \begin{eqnarray*}
 E_0 &=& q^{-2}\Z\\
 E_1 &=& q^0\Z/2 \oplus q^2\Z\\
 E_2 &=& q^2\Z\\
 E_3 &=& q^4 \Z \oplus q^4 \Z/2 \oplus q^6 \Z\\
 E_4 &=& q^6 \Z \oplus q^6 \Z/2 \oplus q^8 \Z\\
 E_n &=& q^4 E_{n-2} \textnormal{ for $n\geq 5$. }
\end{eqnarray*}

If we define $E_n = 0$ for negative $n$ then we see that, when $\alpha =0$,
\[
H_k\left(\!\! \CPic{theta_xxx}\right)=q^{-1} E_k \oplus\bigoplus_{j\geq 1} q^{4j-1}(E_{k-2j}\oplus E_{k-2j}).
\]

Alternatively, we can write

\[
H(E) = \left(q^{-2}+tq^2+t^2q^{2}+\frac{t^3(q^4+q^6)}{1-tq^2}\right)\cdot \Z\oplus \left(t+\frac{t^3 q^4}{1-tq^2}\right)\cdot \Z/2\\
\]
so that
$$ H(\theta_3) = \left( q^{-1} + \frac{2 t^2q^3}{1-q^4 t^2} \right) \cdot H(E). $$

The Poincar\'e series for several families of graphs are provided in the
appendix.

\section{Structural Observations} \label{misc sec}
This section states a number of results on the structure of the chromatic
homology of planar graphs and of the homology of links.  We begin in
\ref{sheet alg} with the analysis of the chain maps from the second
projector to itself. Section \ref{matt sec} states a conjecture on the
structure of the chromatic homology of an arbitrary planar graph. In \ref{2
  colored} the homology of knots is shown to split into an interesting
``unstable'' part, closely related to Khovanov's categorification of the
$2$-colored Jones polynomial, and a periodic ``stable'' portion.

\subsection{Homology of the sheet algebra} \label{sheet alg}
Here we start by analyzing maps between objects in sections \ref{sobmwcat}
and \ref{chromatic cat} (i.e. chain complexes associated to $2$-colored
links and spin networks).  Since the categories are built up from local
pictures, the first interesting example is given by maps between two
intervals. The \emph{sheet algebra} is defined to be the chain complex of
chain maps from the second projector to itself $\Endo(\twoprojector) =
\Morph_{\Kom(2)}(\twoprojector,\twoprojector)$. This forms a differential
graded algebra with differential given by
$$ d_{\twoprojector}(f) = [d,f] = d\circ f + (-1)^{|f|} f\circ d $$

i.e. the graded commutator. The homology of the sheet algebra is homotopy
classes of maps from the projector to itself.

\begin{theorem}
The homology of the sheet algebra with $\Z$ coefficients and $\alpha=0$ is given by
\[
H(\Endo(\twoprojector))=\Z[u]\oplus\Z[u]\cdot  w\oplus \Z[u]/(2u)\cdot b,
\]
as a $\Z[u]$-module. The algebra multiplication is commutative and
determined by $w\cdot b = b^2 = w^2 = 0$. Representatives for the classes
$b$, $u$, and $w$ are given by the chain maps
\[
\begin{array}{ccc}
b &=& (\leftdot, \topdot , \topdot,  \topdot, \topdot, \topdot, \ldots) \\
u &=&  (\isaddle , \turnback, \turnback, \turnback, \turnback,\turnback,\ldots)\\
w &=& (\dottedisaddle, \topdot,\topdot, \topdot,\topdot, \topdot, \ldots)
\end{array}
\]

respectively.  These have homological degree $\deg(b) = 0$, $\deg(u) = -2$ and $\deg(w) = -3$.

\end{theorem}

The proof is by direct computation. Note that $b$ is the class of the ``dotted identity'' $\NWdottedtwoprojector$.  The maps $\NEdottedtwoprojector$,
$\SWdottedtwoprojector$, and $\SEdottedtwoprojector$ are also chain maps,
but they are all homotopic to $\pm \NWdottedtwoprojector$:
\[
\NWdottedtwoprojector\simeq \SWdottedtwoprojector\simeq -\NEdottedtwoprojector\simeq -\SEdottedtwoprojector.
\]

The homology of the sheet algebra is finite dimensional as a module over the
subalgebra generated by $u$. The map $u$ shifts all of the homology down by
two degrees. As a chain map, all of its components are isomorphisms except the first
which is a saddle. The kernel of the map induced by $u$ is the ``unstable''
homology in low degree. The rest of the homology associated to a graph or
knot is called ``stable.'' See sections \ref{matt sec} and \ref{interesting}.

There is an interesting map $R$ from the projector to a rotated projector
given by
\[
\begin{diagram}
\projector & = & \straightthrough & \rTo^{\hsaddle} & q\turnback & \rTo^{\topdot-\bottomdot} & q^3\turnback & \rTo^{\topdot+\bottomdot} & q^5\turnback & \rTo &\cdots\\
\dTo_{R} & =  & \dTo_\hsaddle & & \dTo_\isaddle & & \dTo_\isaddle & & \dTo_\isaddle & &\cdots \\
e^{\frac{\pi i}{2}}\cdot \projector & = & \turnback & \rTo^\isaddle & q\straightthrough & \rTo^{\leftdot-\rightdot} & q^3\straightthrough & \rTo^{\leftdot+\rightdot} & q^5\straightthrough & \rTo & \cdots
\end{diagram}
\]

$R^2$ is a map from the projector to itself which, by neck-cutting, is equal
to $\NWdottedtwoprojector+\SEdottedtwoprojector$. This is homotopic to zero
by the above discussion.  In fact
$$R^2 = dh + hd,$$ where
\[
\begin{diagram}
\projector & = & \straightthrough & \rTo^{\hsaddle} & q\turnback & \rTo^{\topdot-\bottomdot} & q^3\turnback & \rTo^{\topdot+\bottomdot} & q^5\turnback & \rTo &\cdots\\
\dTo_{h} & = &  & \ldTo^{\isaddle} &  &\ldTo^0 &  & \ldTo^\turnback &  & \ldTo^0 &\cdots \\
\projector  & = & \straightthrough & \rTo^{\hsaddle} & q\turnback & \rTo^{\topdot-\bottomdot} & q^3\turnback & \rTo^{\topdot+\bottomdot} & q^5\turnback & \rTo &\cdots .
\end{diagram}
\]
and maps alternate between 0 and $\turnback$. Together, the maps $R$ and $h$
can be used to construct a new differential on the complex formed by pairing
a planar graph $G$ with its dual $\widehat{G}$.

\subsection{A Structural Conjecture} \label{matt sec}

A \emph{cube complex} $C = \bigoplus_{v\in\{0,1\}^n} C_v$ is a chain complex
of diagrams $C_v$ indexed by the vertices of a hypercube $\{0,1\}^n$. For
any vertex $v \in \{0,1\}^n$ set $|v| = \sum_i v_i$.  For any two vertices
$v=(v_1,\ldots,v_n)$ and $w=(w_1,\ldots,w_n)$ we say $v\leq w \textnormal{
  if } v_i\leq w_i$ where $1\leq i \leq n$. If $C$ is a cube complex and $v$
is a vertex then we define the \emph{star} of $v$ in $C$, $\Star_v(C)
\subset C$, to be the subcomplex
$$\Star_v(C) = \bigoplus_{v \leq w} C_w \,.$$

\begin{conjecture}
For every connected planar graph $G$ there exists a cube complex
$C=\bigoplus_{v\in\{0,1\}^n}C_v$ such that
\[
\inp{G}\, \simeq \bigoplus_{v\in\{0,1\}^n}\left(\frac{t^2q^3}{1-t^2q^4}\right)^{|v|}\cdot\Cone^{|v|}\Big(\Star_v(C){\buildrel  \sigma \over\rightarrow} q^2 \Star_v(C)\Big),
\]

where the map $\sigma$ is a handle.
\end{conjecture}

In other words, every chain complex breaks up into a direct sum of
subcomplexes most of which are iterated cones on handle maps. This is
precisely what happens in the computation for the theta graph in section \ref{theta homology}.

\subsection{Structure of the knot invariant} \label{2 colored}

In this section we will discuss the structure of the knot invariant defined
in section \ref{sobmwcat}.

\subsubsection{``Interesting homology is concentrated in low degree''} \label{interesting}
For any knot the homology defined in section \ref{sobmwcat} is necessarily
infinitely generated. However for any two knots we will show that all but a
finite portion of this homology is the same, and the interesting part in low
degree is closely related to the Khovanov categorification of the
$2$-colored Jones polynomial, see \ref{relation to Khovanov}.

Recall that in section \ref{categorified TL section} the dot map was defined
in terms of a handle and the differential $d_n$ of $P_2$ for $n>0$ was defined in
section \ref{second projector section} using sums and differences of these
dot maps. The proposition below implies that these maps do not change up to sign
and homotopy under the ``dotted second Reidemeister move''.

\begin{proposition}{(Handles slide through crossings)}\label{handleslide}
$$\begin{diagram}\CPPic{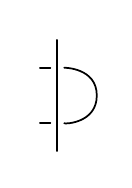} &\rTo^{\pm \CPic{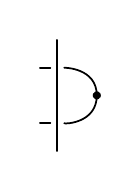}} & \CPPic{handslu} & \hspace{.1in} & \simeq &\hspace{.1in} & \CPPic{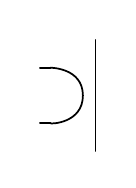} &\rTo^{\mp \CPic{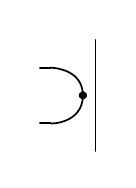}} & \CPPic{handslt} \end{diagram}$$

\end{proposition}

The proof follows from applying the Gaussian Elimination (section
\ref{homotopy lemmas}) twice on the cube obtained by expanding the crossings
on the left hand side above.

\begin{corollary}\label{hcollapse}
The chain complex associated to a framed
knot $K$ in section \ref{sobmwcat} is homotopy equivalent to
$$\begin{diagram}
K^2 &
\rTo^{\hspace{-.2in}\phi} &
q \CPic{circle} &
\rTo^{\hspace{.2in}0\hspace{.2in}} &
q^{3} \CPic{circle} &
\rTo^{\hspace{.2in}2 \MPic{circle-dot}\hspace{.2in}} &
q^{5} \CPic{circle} & \rTo^{\hspace{.2in}0\hspace{.2in}} & \cdots
\end{diagram}$$
where $K^2$ denotes the $2$-cabling of the knot $K$, and the map $\phi$ is
induced by the homotopy, see section \ref{relation to Khovanov} below.
\end{corollary}

\begin{proof}
The first differential in $P_2$ (section \ref{second projector section}) is a
saddle map which turns $K^2$ into the unknot. Using
proposition \ref{handleslide} (applying the Gaussian elimination to the chain complex for $K$)
one slides the end of this unknotted $2$-cabling through the
rest of the knot. The result is pictured above.
\end{proof}

A similar statement may be proved for any link $L$. However, note that the infinite tail for knots, pictured in corollary
\ref{hcollapse}, is standard. When the number of components of $L$ is greater than one this infinite tail will involve
chain complexes for the proper sublinks of $L$.

\subsubsection{Relationship to Khovanov's categorification} \label{relation to Khovanov}

A categorification of the colored Jones polynomials was given in
\cite{CK}. When $n=2$ this construction coincides with the one in section
\ref{sobmwcat}. Here we discuss the relationship between the categorification
above and Khovanov's categorification of the colored Jones polynomial
\cite{Kh1} when $n=2$.

Khovanov defines a chain complex
\[C_{Kh}(K) = \Cone(K^2 \buildrel{\epsilon_*}\over\rightarrow \emptyset)\]

which categorifies the 2-colored Jones polynomial of a framed knot
$K$. $K^2$ is the chain complex which computes the Khovanov homology of the
$2$-cabling of $K$ and $\epsilon_*$ is induced by the $4$-dimensional
cobordism $\epsilon : K^2 \to \emptyset$ obtained by pushing the ribbon bounded by
the $2$-cabling into the $4$-ball.

In order to define $\epsilon_*$ a Morse decomposition of $\epsilon$ must be
chosen. Choose the one in which $\epsilon$ is a composition of a saddle
followed by a disk bounding the resulting unknot. This is an augmentation of
the first two terms of the chain complex in corollary \ref{hcollapse}.  If
we denote these first two terms by $C_{trunc}(K)$ then there is a short
exact sequence

$$0\rightarrow tq^2\Z\rightarrow C_{trunc}(K) \rightarrow C_{Kh}(K)\rightarrow 0$$

Where $tq^2\Z$ is the chain complex consisting only of $\Z$ in bidegree
$(1,2)$. The associated long exact sequence implies that
\[
0\rightarrow H_{trunc}^0(K)\rightarrow H^0_{Kh}(K)\rightarrow q^2\Z\rightarrow H_{trunc}^1(K)\rightarrow H^1_{Kh}(K)\rightarrow 0
\]
and $H_{trunc}^i(K)\cong H_{Kh}^i(K)$ for $i\neq 0,1$.

\section{Appendix: Computations for graphs and links} \label{highertheta}

The homology is given for certain families of graphs, and for some examples of links.

\subsection{Chromatic Homologies of trees and cycles}

\begin{align*}
\theta_n &= \begin{minipage}{.4in}
\includegraphics[scale=.4]{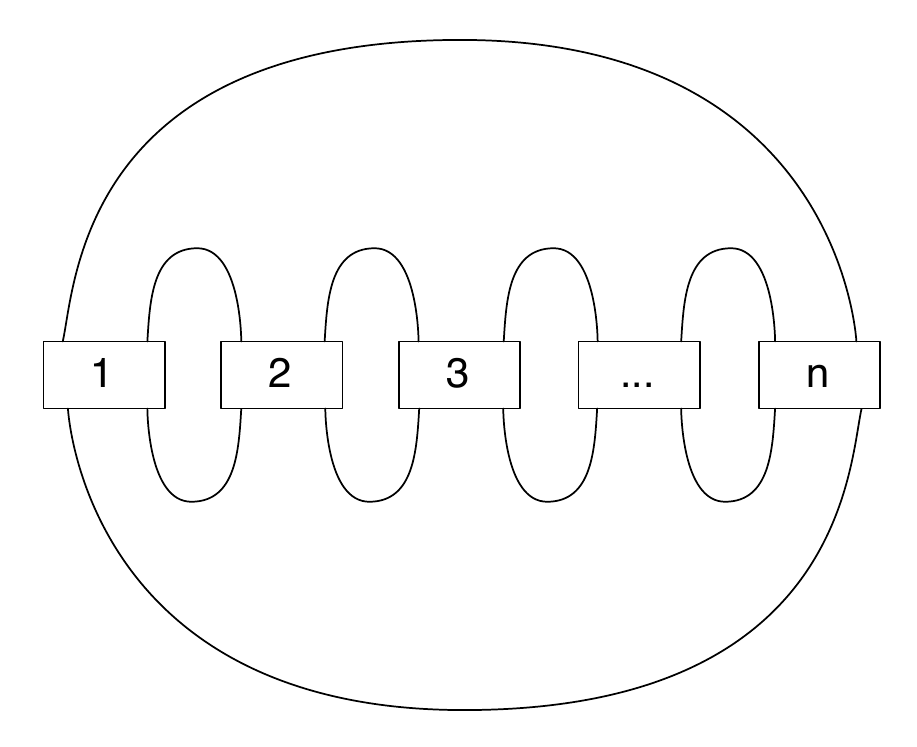}
\end{minipage}
&
T_n &= \begin{minipage}{.4in}
\includegraphics[scale=.4]{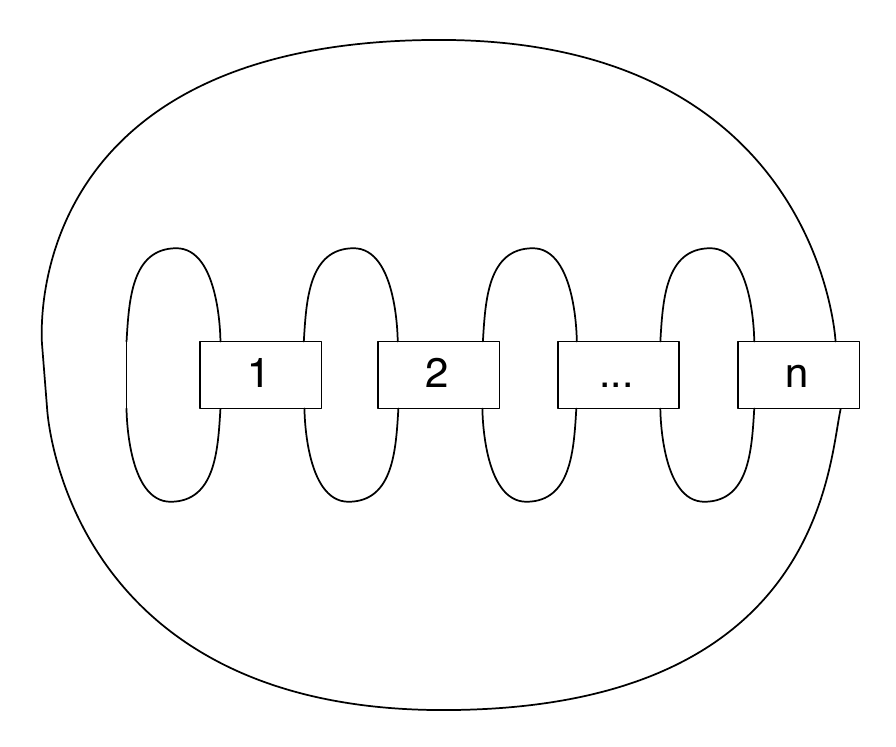}
\end{minipage} \\
T_1 &= \cTrace &  E &= \cTrace\;\rightarrow q^2 \cTrace\\
\end{align*}

Let $\theta_n$ denote the $n$th theta graph, with two vertices and $n$ edges connecting them. 
$\theta_n$ is dual to a cycle with $n$
edges or the boundary of an $n$-gon (see illustration in \ref{howdy}). 
The graph $T_n$ is dual to the graph with a single
vertex and $n$ loops. $\langle T_n\rangle$ computes the chromatic homology of any tree with $n$ edges.
The homology of the trace $T_1 = \tr(P_2)$ (section \ref{homology
  subsection}) and the chain complex $E$ (section \ref{theta homology}) are
used below to express the homologies of $T_n$ and $\theta_n$. Using $\Z$
coefficients and $\alpha = 0$ we have

\begin{eqnarray*}
H(T_n) & = & q^{1-n} H(T_1)\oplus \frac{q^{1-n}}{1+x}\left[\left(1+\frac{x^2}{1-x}\right)^{n-1}\!\! - 1\right]\bigg|_{x=tq^2}\cdot H(E),
\\
H(\theta_n) & = & H(B_n) \oplus \frac{q^{2-n}}{1-x^2}\left[\left(1+\frac{x^2}{1-x}\right)^{n-1}\!\!-\left(x+\frac{x^2}{1-x}\right)^{n-1}\!\! + x^{n-1} - 1\right]\bigg|_{x=tq^2} \cdot H(E),
\end{eqnarray*}

$$\textnormal{\quad where:\quad} H(B_n) = \begin{cases}
q^{1-2k}\frac{1-x^{2k}}{1-x^2}\Big|_{x=tq^2} \cdot H(E) & \text{ for $n=2k+1$}\\
q^{2-2k} H(T_1) \oplus q^{2-2k}\frac{x-x^{2k-1}}{1-x^2}\Big|_{x=tq^2} \cdot H(E) & \text{ for $n=2k$}
\end{cases}$$

and

\begin{eqnarray*}
H(E) &=& \left(q^{-2}+tq^2+t^2q^{2}+\frac{t^3 (q^4+q^6)}{1-tq^2}\right)\cdot \Z\oplus \left(t+\frac{t^3q^4}{1-tq^2}\right)\cdot \Z/2\\
H(T_1) &=& \left(q^{-2} + 1 +\frac{t^2q^2+t^3q^6}{1-t^2q^4}\right)\cdot \Z\oplus \left(\frac{t^3q^4}{1-t^2q^4}\right)\cdot \Z/2.
\end{eqnarray*}

\subsection{Knots and links.}

If $2_1^2$, $3_1$ and $4_1$ denote the Hopf link, the positively oriented
trefoil and figure eight knots respectively then their homologies have been
computed,

\begin{eqnarray*}
H(2_1^2) &=&(t^{-4}(q^{-8}+q^{-6})+t^{-2}q^{-4}+t\inv+(1+q^{-2})+t(1+q^2+q^4)\\
& & + \,\, t^2 q^4 +t^4(q^6+q^8) )\cdot \Z + (t\inv q^{-2}+q^2 t^2)\cdot \Z/2\\
& & +\,\, H(T_1)^2 - (q^{-2} + 1)^2\cdot \mathbb{Z}\\
H(3_1) &=& (t^{-6}(q^{-10}+q^{-8})+t^{-4}q^{-6}+t^{-3}q^{-2}+t^{-2}(q^{-4}+q^{-2})+t\inv(1+q^2)\\
& & +\,\, (1+q^{-2})+t(q^2+q^4)+ t^2q^2+t^3q^6+ t^5q^8+t^6q^{12})\cdot\Z\\
& & +\,\,(t^{-3}q^{-4}+t(1+q^2)+ t^3(2 q^4+q^6)+t^4(q^6+q^8)+t^6q^{10})\cdot \Z/2\\
& & +\,\, H(T_1) - (q^{-2} + 1) \cdot \Z \\
H(4_1) &=&
(t^{-8} q^{-14} +
t^{-7} q^{-10} +
t^{-5} q^{-8} +
t^{-4} (q^{-8} + q^{-4}) +
t^{-3} (q^{-6} + q^{-4}) \\ & & \quad +\,\,
t^{-2} (q^{-6} + q^{-4} + q^{-2}) +
t^{-1}(q^{-4} + 2 q^{-2} + 1) +
(2 q^{-2} + 3 + q^2)  \\ & & \quad +\,\,
t (1 + 2 q^2 + q^4) +
t^2 (q^2 + q^4 + q^6) +
t^3 (q^4 + q^6) +
t^4 (q^4 + q^8) +
t^5 q^8  \\ & & \quad +\,\,
t^7 q^{10} +
t^8 q^{14}) \cdot \Z  \\ & & \quad +\,\,
(t^{-7} q^{-12} +
t^{-5} (q^{-10} + q^{-8}) +
t^{-4} (q^{-8} + 2 q^{-6}) +
t^{-3} q^{-6} +
t^{-2} (2q^{-4} + q^{-2})  \\ & & \quad +\,\,
t^{-1} (q^{-4}+ 2q^{-2}) +
(q^{-2} + 1) +
t (1 + q^2) +
t^2 (1 + 2q^2 + q^4) +
t^3 (q^2 + 2q^4)  \\ & & \quad +\,\,
t^4 q^6 +
t^5 (2 q^6 + q^8) +
t^6 (q^8 + q^{10}) +
t^8 q^{12}) \cdot \Z/2 \\ & & +\,\, H(T_1) - (q^{-2} + 1) \cdot \Z \\
\end{eqnarray*}

The $H(T_1) - (q^{-2} + 1)\cdot \Z$ term is the infinite tail, see section
\ref{2 colored}. Notice in $4_1$ that the free part of the homology is
symmetric away from homological degree 0. The missing $q^2$ term can be
found in homological degree 2 of the infinite tail, giving a symmetric
graded Euler characteristic.  This was computed using the JavaKh program
written by Jeremy Green and Scott Morrison \cite{GM}.

{\bf Acknowledgements.} V. Krushkal was supported in part by NSF grant DMS-1007342 and by the I.H.E.S.

\end{document}